\documentclass{amsart}
\usepackage{amsmath,amssymb,amsthm,amsopn,stmaryrd}
\usepackage[all]{xy}
\numberwithin{equation}{section}

\newtheorem{thm}[equation]{Theorem}
\newtheorem{cor}[equation]{Corollary}
\newtheorem{prop}[equation]{Proposition}
\newtheorem{lem}[equation]{Lemma}
\theoremstyle{definition}

\theoremstyle{remark}
\newtheorem{rmk}[equation]{Remark}

\def\co{\colon\thinspace}
\newcommand{\mb}[1]{\mathbb{#1}}
\newcommand{\mf}[1]{\mathfrak{#1}}
\newcommand{\mc}[1]{\mathcal{#1}}
\newcommand{\mit}[1]{\mathit{#1}}
\newcommand{\mr}[1]{\mathrm{#1}}
\newcommand{\mbf}[1]{\mathbf{#1}}

\newcommand{\Hom}{\ensuremath{{\rm Hom}}}

\newcommand{\overfrom}{\mathop\leftarrow}

\newcommand{\TAF}{{\rm TAF}}

\newcommand{\End}{{\rm End}}
\newcommand{\Aut}{{\rm Aut}}

\newcommand{\Spec}{{\rm Spec}}
\newcommand{\Spf}{{\rm Spf}}

\newcommand{\TMF}{{\rm TMF}}
\newcommand{\Sh}{{\rm Sh}}

\newcommand{\cF}{\bar{\mb F}}

\newcommand{\QQ}{\mathbb{Q}}
\newcommand{\AF}{\mathbb{A}}
\newcommand{\ZZ}{\mathbb{Z}}
\newcommand{\GG}{\mathbb{G}}
\newcommand{\FF}{\mathbb{F}}
\newcommand{\CC}{\mathbb{C}}
\newcommand{\RR}{\mathbb{R}}
\newcommand{\bra}[1]{\langle#1\rangle}

\DeclareMathOperator{\Tr}{Tr}
\DeclareMathOperator{\Cl}{Cl}
\DeclareMathOperator{\Gal}{Gal}
\DeclareMathOperator{\Lie}{Lie}

\newcommand{\comp}[1]{\ensuremath{#1^\wedge}}

\newcommand{\mmod}{{\sslash}}

\newcommand{\xym}[1]{
\vskip 0.7pc
\centerline{\xymatrix{#1}}
\vskip 0.7pc
}

\title{Topological automorphic forms on $U(1,1)$}

\author{Mark Behrens}
\address{Department of Mathematics, Massachusetts Institute of Technology, 
Cambridge, MA 02140}
\email{mbehrens@math.mit.edu}
\thanks{The first author was partially supported by NSF Grant 
\#0605100, a grant from the Sloan
foundation, and DARPA}

\author{Tyler Lawson}
\address{Department of Mathematics, University of Minnesota, 
Minneapolis, MN 55455}
\email{tlawson@math.umn.edu}
\thanks{The second author was partially supported by NSF Grant \#0805833}

\date{\today}

\begin{document}

\begin{abstract}
The homotopy type and homotopy groups of some spectra 
$\TAF_{GU}$ of topological
automorphic forms associated to a unitary similitude group $GU$ of type $(1,1)$
are explicitly described in quasi-split cases.  The
spectrum $\TAF_{GU}$ is shown to be closely related to the spectrum $\TMF$
in these cases, and homotopy groups of some of these spectra are explicitly
computed.    
\end{abstract}

\maketitle

\section{Introduction}

Let $F$ be a quadratic imaginary extension of $\QQ$, and $p$ be a fixed
prime which splits as $u \bar{u}$ in $F$.
Let $GU$ be the
group of unitary similitudes of a Hermitian form on an $F$-vector space of
signature $(1,n-1)$, and suppose that $K \subset GU(\AF^{p,\infty})$ 
is a compact open subgroup of
the finite adele points of $GU$ away from $p$.  Associated to this data is
a unitary Shimura variety $\Sh(K)$ whose complex points are given by an
adelic quotient
$$ \Sh(K)(\CC) \cong GU(\QQ) \setminus GU(\AF) / K\cdot K_p
\cdot K_\infty. $$
Here $K_p \subset GU(\QQ_p)$ is maximal compact, and $K_\infty \subset
GU(\RR)$ is maximal compact modulo center.  The Shimura variety admits
a $p$-integral model as a moduli stack of certain $n$-dimensional polarized 
abelian varieties $A$ with
complex multiplication by $F$ and level structure dependent on $K$.  
Over a $p$-complete ring, the complex
multiplication decomposes the formal completion of $A$ as
$$ \widehat{A} \cong \widehat{A}_u \oplus \widehat{A}_{\bar{u}} $$
and the summand $\widehat{A}_u$ is required to be $1$-dimensional.  We
refer the reader to \cite{HarrisTaylor}, \cite{kottwitz}, and \cite{taf}
for a detailed exposition of these moduli stacks.

In \cite{taf}, the authors used a theorem of Jacob Lurie to
construct  a $p$-complete $E_\infty$-ring spectrum
$\TAF_{GU}(K)$ of topological automorphic forms ``realizing'' the 
formal groups
$\widehat{A}_u$ associated to the Shimura
varieties $\Sh(K)$.  The height of the formal groups $\widehat{A}_u$ varies
between $1$ and $n$, and the resulting spectra are $v_n$-periodic.  
In \cite[Ch.~15]{taf}, the authors showed that in the
case of $n=1$, for appropriate choices of $GU$ and $K$, the associated
spectrum of topological automorphic forms is essentially a product of copies of $K$-theory indexed by the
class group of $F$.  Thus the $n=1$ case of the theory reduces to a
well-understood $v_1$-periodic cohomology theory. 

The purpose of this paper is to demystify the case of $n=2$, and show that, at
least in some cases, the resulting cohomology theory $\TAF_{GU}(K)$ is
closely related to $\TMF$, the Goerss-Hopkins-Miller theory of topological
modular forms.  Thus in these cases, the $n=2$ case of the theory of topological
automorphic forms reduces to the well-understood theory of topological
modular forms in a manner analogous to the way in which the $n = 1$ theory
reduces to $K$-theory.

In the case of $n = 2$, the $p$-completion of the 
associated moduli stack $\Sh(K)$ consists of
certain polarized abelian surfaces $A$ with complex multiplication by $F$ with
$\dim \widehat{A}_u = 1$.  In Section~\ref{sec:Honda-Tate}, we use the
Honda-Tate classification of isogeny classes of abelian varieties over
$\bar{\FF}_p$ to show that every $F$-linear abelian surface $A$ with $\dim
A_u = 1$ is isogenous to a product of elliptic curves, and that there is a
bijective correspondence between such isogeny classes of abelian surfaces
and the isogeny classes of elliptic curves.  Although this material serves
as motivation for the constructions of the later sections, the remainder of
the paper is independent of Section~\ref{sec:Honda-Tate}.

We concentrate solely on the case where the Hermitian form is
isotropic, and study the case of two compact open subgroups $K_0$ and $K_1$
that are, in some sense, extremal examples of maximal compact open
subgroups of $GU(\AF^{p,\infty})$.  A
detailed account of the initial data, the moduli functor represented
by $\Sh(K)$, and the associated cohomology theory $\TAF_{GU}(K)$ 
is given in Section~\ref{sec:moduli}.

Section~\ref{sec:tensor} gives a description of the moduli stack
$\Sh(K_0)$.
The
subgroup $K_0$ is the stabilizer of a non-self-dual lattice chosen such
that the moduli stack $\Sh(K_0)$ admits a complete uniformization by copies
of the moduli stack $\mc{M}_{\mit{ell}}$ of elliptic curves.  We show that there is an
equivalence:
$$ \coprod_{\Cl(F)} \mc{M}_{\mit{ell},\ZZ_p} \cong \Sh(K_0) 
 \quad \quad \text{(Theorem~\ref{thm:etale})}
$$
except in the cases where $F = \QQ(i)$ or
$\QQ(\omega)$, where a slight modification is given.  In
Section~\ref{sec:elliptictensor}, the associated
spectra of topological automorphic forms are computed to be 
$$ \TAF_{GU}(K_0) \simeq \prod_{\Cl(F)} \TMF_p 
 \quad \quad \text{(Theorem~\ref{thm:tensor})}
$$
except in the cases mentioned above.  These cases are analyzed separately.

In Section~\ref{sec:quotient} we study the moduli stack $\Sh(K_1)$.  The
subgroup $K_1$ is defined to be the stabilizer of a self-dual lattice, and
the resulting moduli stack may be taken to be the moduli stack of
\emph{principally} polarized abelian surfaces with complex multiplication by
$F$.  We are only able to give a description of a connected component
$\Sh(K_1)_0$ of
$\Sh(K_1)$.  Let $N$ be such that $F = \QQ(\sqrt{-N})$, and let
$\mc{M}_0(N)$ be the moduli stack of elliptic curves with
$\Gamma_0(N)$-structure.  We show that there
is an equivalence 
$$ \mc{M}_0(N)_{\ZZ_p} \mmod \bra{w} \cong \Sh(K_1)_0 
 \quad \quad \text{(Theorem~\ref{thm:quotient})}
$$
where $w$ is the Fricke involution, unless $N = 1$ or $3$, where slightly
different descriptions must be given.  The homotopy groups of the
corresponding summand of the spectrum $\TAF_{GU}(K_1)$ is analyzed in
Section~\ref{sec:ellipticquotient}.  In general, one takes the fixed points
of modular forms for $\Gamma_0(N)$ 
with respect to an involution.  Complete
computations are given in the cases $N = 1,2,3$.

\section{Honda-Tate theory of $F$-linear abelian
surfaces}\label{sec:Honda-Tate}

In this section we analyze the isogeny classes of abelian
surfaces $A$ over $\bar{\FF}_p$ with complex multiplication
$$ i : F \rightarrow \End^0(A), $$
using Honda-Tate theory.
The splitting
$$ {\mc O}_{F,p} \cong \ZZ_p \times \ZZ_p $$
induces a splitting of the formal group
$$ \widehat{A} \cong \widehat{A}_u \times \widehat{A}_{\bar{u}}. $$
We will always assume $\dim \widehat{A}_u = 1$.  This is equivalent to
assuming that the summand $A(u)$ of the $p$-divisible group $A(p)$ is
$1$-dimensional.   
The analysis of this section is independent of
the rest of the paper, but serves to motivate some of the constructions in
later sections.

We recall from \cite[Theorem 2.2.3]{taf} that the Honda-Tate
classification of simple abelian varieties implies that
isogeny classes of simple abelian
varieties over $\cF_p$ are in one-to-one correspondence with {\em
minimal $p$-adic types}.

Let $M$ be a CM field (a field with a complex conjugation $c$ whose
fixed field is totally real), and for any prime $x$ over $p$ we let
$f_x$ be the degree of the residue field extension and $e_x$ the
ramification index.  A $p$-adic type $(M,(\eta_x))$ consists of such a
CM field, together with positive rational numbers $\eta_x$ for all
primes $x$ over $p$ of $M$, satisfying the relation
\[
\eta_x /e_x + \eta_{c(x)} /e_x = 1
\]
for all $x$.  The associated simple abelian variety $A$ 
has $M = \mr{center}(\End^0(A))$ and dimension
$\frac{1}{2} [M:\mb Q] m$, where $m = [\End^0(A):M]^{1/2}$.
The $p$-divisible group of $A$ breaks up as the
sum of simple $p$-divisible groups $A(x)$, each with height $[M_x: \QQ_x]m$,
dimension $\eta_x f_x m$, and pure slope $\eta_x/e_x$.  
The $p$-completion of the
endomorphism ring of $A$ is a product over $x$ of division algebras
$\End^0(A(x))$ with center $M_x$ and invariant $\eta_x f_x$.

A map of CM fields $M \hookrightarrow M'$ takes a $p$-adic type $(M,(\eta_x))$
to $(M',(\eta_x e_{x'/x}))$, where $e_{x'/x}$ is the ramification degree of
the prime $x'$ over the prime $x$.  The $p$-adic type is minimal if it
is not in the image of such a non-identity map.

Simple $F$-linear abelian varieties are classified by initial objects of
the subcategory of $p$-adic types 
$(L, (\eta_x))$ under $F$.  Such $F$-linear
abelian varieties are isotypical, with simple type given by the minimal
$p$-adic type over $L$.  The associated simple $F$-linear 
abelian variety $A$ has $L = \mr{center}(\End^0_F(A))$ and
dimension
$\frac{1}{2} [L:\mb Q] t$, where $t = [\End^0_F(A):L]^{1/2}$.
The $p$-completion of the
endomorphism ring of $A$ is a product over $x$ of division algebras
$\End_{F_x}(A(x))$ with center $L_x$ and invariant $\eta_x f_x$.

Let $E$ be an elliptic curve over $\bar{\FF}_p$.  
Choosing a basis of $\mc{O}_F$ gives an inclusion
$$ \mc{O}_F \hookrightarrow M_2(\ZZ). $$
We associate to $E$ an
$F$-linear abelian surface
$$ E \otimes \mc{O}_F := E \times E $$
with complex multiplication given by the composite
$$ \mc{O}_F \hookrightarrow M_2(\ZZ) \hookrightarrow M_2(\End(E)) \cong \End(E
\times E). $$

We now classify the isogeny classes $[(A,i)]$ 
of $F$-linear abelian surfaces $(A,i)$
with $\dim \widehat{A}_u = 1$. 

\subsection*{Case 1: $(A,i)$ is simple.}

Associated to $(A,i)$ is a minimal $p$-adic type $(L,\eta)$ under $F$ with 
$$ 2 = \dim A = \frac{1}{2} [L: \QQ] t. $$
Since $L$ contains $F$, $[L:\QQ]$ is divisible  by $2$.
We therefore have two possibilities, corresponding to $t = 1$ and $t = 2$.

\subsubsection*{Subcase 1a: $t = 1$.}

In this case, $L = \End_F^0(A)$ is a 
CM-field which is quadratic extension of $F$.  In particular, $L$ is
totally complex, and posesses a unique involution $c$ for which $L^{\bra{c}}$ 
is totally real, and which restricts to conjugation on $F$.  Since $L$ is a
quadratic extension of $F$, there is also an involution $\sigma$ of $L$
satisfying $L^{\bra{\sigma}} = F$.  Since they give distinct fixed fields, the
involutions $\sigma$ and $c$ must be distinct, and we deduce that $L$ is
Galois over $\QQ$ with Galois group  
$$ \Gal(L/\QQ) \cong C_2 \times C_2 = \bra{ c, \sigma}. $$
The prime
$u$ of $F$ lying over $p$ is either split, ramified, or inert in $L$. 
It is easy to see that if $p$ is ramified or inert, any $p$-adic type
$\eta$ associated to $L$ must come from one on $F$, and so will not be
minimal under $F$.  Therefore, for $(L,\eta)$ to be minimal under $F$, $u$
must split as $v \sigma(v)$ in $L$.  
Then $\bar{u}$ splits as
$c(v) c\sigma(v)$.  We must have
$$ \eta_{\sigma^\epsilon v} + \eta_{c(\sigma^{\epsilon}v)} = 1 $$
for $\epsilon \in \{ 0,1\}$.
Since $A$ is $2$-dimensional, we deduce $\eta_{\sigma^\epsilon v} 
\in \{ 0, 1 \}$.  Since we are assuming $\dim A(u) = 1$, and 
$A(u) = A(v_1) \oplus A(v_2)$, one of the
$\eta_{\sigma^\epsilon v}$ must equal $1$ and the other must be $0$.  
Without loss of
generality, assume $\eta_{v} = 1$.  It follows that we must have
\begin{align*}
\eta_{c(v)} & = 0, \\
\eta_{\sigma(v)} & = 0, \\
\eta_{\sigma c(v)} & = 1. \\
\end{align*}

Although $L$ is a minimal $p$-adic type under $F$, it is not minimal under
$\QQ$.  Indeed, letting $F'$ be the quadratic imaginary of $\QQ$ given by 
$L^{\sigma c = 1}$, with conjugation 
$$ c' = c \vert_{F'} = \sigma \vert_{F'}, $$ 
the prime $p$ must split as $w c'(w)$ in $F'$.  The prime $w$ splits as $v
\sigma c (v)$ in $L$, and the prime $c'(w)$ splits as $c(v)\sigma(v)$ in
$L$.
The $p$-adic type $(L,\eta)$
is induced from a $p$-adic type $(F', \eta')$, where:
\begin{align*}
\eta'_{w} & = 1, \\
\eta'_{c'(w)} & = 0. \\
\end{align*}
The isogeny class of 
simple abelian varieties associated to $(F', \eta')$ is an elliptic curve $E$
with complex multiplication by $F'$
such that $E(w)$ is $1$-dimensional, and the isogeny class of $F$-linear
abelian varieties containing $(A,i)$ is
given by
$$ [(A,i)]  = [E \otimes \mc{O}_F]. $$ 

\subsubsection*{Subcase 1b: $t = 2$.}

In this case $L = F$.  Since we are assuming $\dim A(u) = 1$, we deduce
that $\eta_u = 1/2$.  We must therefore have $\eta_{\bar{u}} = 1/2$.

The $p$-adic type $(F,\eta)$ is not minimal under $\QQ$: it is induced from
the $p$-adic type $(\QQ, \eta')$ where $\eta_p = 1/2$.  The 
isogeny class corresponding to $(\QQ, \eta')$ is the isogeny class of a 
supersingular elliptic curve $E$.  Thus the $F$-linear isogeny class $(A,i)$
contains $E \otimes \mc{O}_F$ as a representative.

\subsection*{Case 2: $(A,i)$ is not simple.}

Then we must have
$$ [(A,i)] = [(A_1, i_1)] \oplus [(A_2, i_2)] $$
where $(A_j, i_j)$ are $1$-dimensional $F$-linear abelian varieties.  Thus
the abelian varieties $A_j$ are elliptic curves with complex multiplication
by $F$.  There is one isogeny class $[E]$ of elliptic curves with complex
multiplication by $F$, and a representative $E$ admits $2$ conjugate
complex multiplications.  There result two distinct $F$-linear isogeny
classes:
$$ [(E, i_0)] \quad \text{and} \quad [(\bar{E}, \bar{i}_0)]. $$
Here we use $E$ to denote the $F$-linear elliptic curve with $\dim E(u) =
1$, and $\bar{E}$ to be the same elliptic curve with conjugate complex
multiplication, so that $\dim \bar{E}(u) = 0$.  Since $\dim A(u) = 1$, we must
have an $F$-linear isogeny $A \simeq E \times \bar{E}$.  The diagonal
embedding of abelian varieties
$$ E \hookrightarrow E \times \bar{E} $$
extends to an $F$-linear quasi-isogeny
$$ E \otimes \mc{O}_F \xrightarrow{\simeq} E \times \bar{E} $$
where $\mc{O}_F$ acts on $E \otimes \mc{O}_F$ on the second factor only.
Thus the isogeny class is computed to be $[(A,i)] = [E \otimes \mc{O}_F]$.

\begin{prop}
The $F$-linear isogeny classes of $(A,i)$ over $\bar{\FF}_p$ 
with $\dim A(u) = 1$ are given by
\begin{description}
\item[Case 1a] $[E\otimes \mc{O}_F]$, where $E$ is an elliptic curve with
complex multiplication by a quadratic imaginary extension $F' \ne F$ in
which $p$ splits. 
\item[Case 1b] $[E\otimes \mc{O}_F]$, where $E$ is a supersingular 
elliptic curve. 
\item[Case 2] $[E \otimes \mc{O}_F]$, where $E$ is an elliptic curve with
complex multiplication by $F$.
\end{description}
\end{prop}

\begin{cor}
The construction
$$ [E] \mapsto [E \otimes \mc{O}_F] $$
gives a bijection between isogeny classes of elliptic curves over
$\bar{\FF}_p$, and isogeny
classes of $F$-linear abelian surfaces $A$ over $\bar{\FF}_p$ 
with $\dim A(u) = 1$.
\end{cor}

\section{Overview of the Shimura stack}
\label{sec:moduli}

We first review the moduli problem represented by the Shimura stacks under
consideration.  A more complete description, with motivation, can be found
in \cite{taf}. Fix a prime $p$ and consider the following initial data: 
\begin{align*}
F  = & \: \text{quadratic imaginary extension of $\QQ$ in which $p$ splits as
$u\bar{u}$,} \\
\mc{O}_F = \: & \text{ring of integers of $F$,} \\
V = \: & \text{$F$-vector space of dimension $2$,} \\
\bra{-,-} = \: & \text{$\QQ$-valued non-degenerate hermitian alternating
form.}
\end{align*}
We require that, for a complex embedding of $F$, 
the signature of $\bra{-,-}$ on $V$ is $(1,1)$. When necessary, we will
regard 
$$ F = \QQ(\delta) \quad \text{where} \quad \delta^2 = -N $$ 
for a positive square-free integer $N$. 
Let $\iota$ denote the involution on
$\End_F(V)$, defined by $\bra{\alpha v, w} = \bra{v, \alpha^\iota w}$.

Let $GU = GU_V$ be the associated unitary similitude group over $\QQ$, 
with $R$-points
\begin{align*}
GU(R) = & \{ g \in \End_F(V) \otimes_\QQ R \: : \: \bra{gv, gw} = \nu(g)
\bra{v,w}, \: \nu(g) \in R^\times \} \\
= & \{ g \in \End_F(V) \otimes_\QQ R \: : \: g^\iota g \in R^\times \}.
\end{align*}

Let $\widehat{\ZZ}^p$ denote the product $\prod_{\ell \ne p} \ZZ_\ell$,
and let $\AF^{p,\infty}$ be the finite adeles away from $p$, so that we
have
$$ \AF^{p,\infty} = \widehat{\ZZ}^p \otimes \QQ. $$
We let $V^{p,\infty}$ denote $V \otimes \AF^{p,\infty}$.
For an abelian variety $A/k$, where $k$ is an algebraically closed field of
characteristic $0$ or $p$, we define
$$ V^p(A) = \left(\prod_{\ell \ne p} T_\ell(A) \right) \otimes \QQ $$
where 
$$ T_\ell(A) = \varprojlim_{i} A(k)[\ell^i] $$
denotes the covariant $\ell$-adic Tate module of $A$.

For every compact open subgroup 
$$ K \subset GU(\AF^{p,\infty}) $$
there is a Deligne-Mumford stack $\Sh(K)/ \Spec(\ZZ_p)$.  For a
locally noetherian connected $\ZZ_p$-scheme $S$,
and a geometric point $s$ of $S$,
the $S$-points of $\Sh(K)$ are the groupoid whose objects are tuples
$(A,i,\lambda,[\eta]_{K})$, with:
\begin{center}
\begin{tabular}{lp{19pc}}
$A$, & an abelian scheme over $S$ of dimension $2$,\\
$\lambda\co A \rightarrow A^\vee$, & a $\ZZ_{(p)}$-polarization, \\
$i\co \mc{O}_{F,(p)} \hookrightarrow \End(A)_{(p)}$, & an inclusion of
rings, such that the $\lambda$-Rosati involution is compatible with
conjugation, \\
$[\eta]_K$, & a $\pi_1(S,s)$-invariant $K$-orbit
of $F$-linear similitudes: \\
& \qquad $ \eta\co (V^{p,\infty}, \bra{-,-}) \xrightarrow{\cong}
(V^p(A_s), \bra{-,-}_\lambda), $
\end{tabular}
\end{center}
subject to the following condition:
\begin{equation}\label{eq:condition}
\text{the coherent sheaf $\Lie A \otimes_{\mc{O}_{F,p}} \mc{O}_{F,u}$ is
locally free of rank $1$.}
\end{equation}
Here, since $S$ is a $\ZZ_p$-scheme, the action of $\mc{O}_{F,(p)}$ on
$\Lie A$
factors through the $p$-completion $\mc{O}_{F,p}$.

The morphisms 
$$ (A, i, \lambda, \eta) \rightarrow (A', i', \lambda',\eta') $$
of the groupoid of $S$-points of $\Sh(K)$ are the
prime-to-$p$ quasi-isogenies of abelian schemes 
$$ \alpha\co A \xrightarrow{\simeq} A' $$
such that
\begin{alignat*}{2}
\lambda & = r\alpha^\vee \lambda' \alpha, 
\quad && 
r \in \ZZ_{(p)}^\times, \\
i'(z)\alpha & = \alpha i(z), 
\quad &&
z \in \mc{O}_{F,(p)}, \\
[\eta']_K & = [\eta \circ \alpha_*]_K.
\quad &&
\end{alignat*}

The $p$-completion $\Sh(K)^\wedge_p/\Spf(\ZZ_p)$ is determined by the
$S$-points of $\Sh(K)$ on which $p$ is locally nilpotent.  On such schemes,
the abelian surface $A$ 
has a $2$-dimensional, height $4$
$p$-divisible group $A(p)$, and the composite 
$$ \mc{ O}_{F,(p)} \xrightarrow{i} \End(A)_{(p)}
\to \End(A(p)) $$ 
factors through the $p$-completion
\[
\mc{O}_{F,p} \cong \mc{O}_{F,u} \times \mc{O}_{F,\bar u} \cong 
\mb Z_p \times \mb Z_p.
\]
Therefore, the action of $\mc{ O}_F$ naturally splits $A(p)$ into two
summands, $A(u)$ and $A(\bar u)$, both of height $2$.  For such schemes
$S$, Condition~(\ref{eq:condition}) is equivalent to the condition that
$A(u)$ is $1$-dimensional.  This forces the formal group of $A$ to split
into two $1$-dimensional formal summands.

\begin{rmk}
  We pause to relate this to the moduli discussed in \cite{taf}.
  There, the moduli at chromatic height $2$ actually consisted of
  $4$-dimensional abelian varieties with an action of an order $\mc{O}_B$
  in a
  $2$-dimensional central simple algebra $B$ over $F$ with involution of
  the second kind.  The moduli problem
  we consider here is the case where this order is $\mc{O}_B = 
  M_2(\mc{ O}_F)$,
  with involution being conjugate-transpose; any abelian variety $A$
  with such an action is canonically isomorphic to $A_0^2$ for some
  abelian surface $A_0$ with complex multiplication by $\mc{ O}_F$, 
  together with a
  polarization $A^2_0 \to (A^\vee)_0^2$ which is a product of two copies
  of the same polarization.
\end{rmk}

A theorem of Jacob Lurie \cite[Thm.~8.1.4]{taf} associates to a
$1$-dimensional $p$-divisible group $\GG$ over a 
locally noetherian separated Deligne-Mumford stack
$X/\mathrm{Spec}(\ZZ_p)$ which is locally a universal deformation of all of its
mod $p$ points, a (Jardine fibrant) presheaf
of $E_\infty$-ring spectra $\mc{E}_\GG$ on the site
$(X^\wedge_p)_{et}$.  
The construction is functorial in $(X,\GG)$: given another pair
$(X',\GG')$, a morphism $g\co X \rightarrow X'$ of Deligne-Mumford stacks 
over $\Spec(A)$, 
and an isomorphism of $p$-divisible groups
$\alpha\co \GG
\xrightarrow{\cong} g^* \GG'$, there is an induced map of presheaves of 
$E_\infty$-ring spectra
$$ (g,\alpha)^*\co g^*\mc{E}_{\GG'} \rightarrow \mc{E}_{\GG}. $$
A proof of this theorem has not yet appeared in print.

If $(\mbf{A}, \mbf{i}, \pmb{\lambda}, [\pmb{\eta}])$
is the universal tuple over $\Sh(K)$, then the $p$-divisible group
$\mbf{A}(u)$ satisfies the hypotheses of Lurie's theorem \cite[Sec.~8.3]{taf}.  The 
associated sheaf will be
denoted
$$ \mc{E}_{GU} := \mc{E}_{\mbf{A}(u)}. $$
The $E_\infty$-ring spectrum of topological automorphic forms is obtained
by taking the global sections:
$$ \TAF_{GU}(K) := \mc{E}_{GU}(\Sh(K)^\wedge_p). $$

For the remainder of this paper, we fix $V = F^2$, with alternating form:
$$ \bra{(x_1, x_2), (y_1, y_2)} = \Tr_{F/\QQ} \left( 
\begin{bmatrix} x_1 & x_2 \end{bmatrix}
\begin{bmatrix}0 & -1 \\ 1 & 0 \end{bmatrix}
\begin{bmatrix}\bar y_1 \\ \bar y_2\end{bmatrix} \right). $$
Let $GU = GU_V$ be the associated unitary similitude group.  This will be
the only case we shall consider in this paper.

\section{A tensoring construction}\label{sec:tensor}

In this section, let 
$L$ be the lattice $\mc{O}_F^2 \subset F^2 = V$, and let $K_0$ denote
the compact open subgroup of $GU(\AF^{p,\infty})$ given by
$$ K_0 = \{ g \in GU(\AF^{p,\infty}) \: : \: g(\widehat{L}^p) =
\widehat{L}^p \} $$
where $\widehat{L}^p := L \otimes \widehat{\ZZ}^p$.
In this section we will show that the associated Shimura variety $\Sh(K_0)$
is closely related to the 
moduli stack of elliptic curves.

Let $V_0 = \QQ^2$.
Let $\mc{M}_{\mathit{ell},\ZZ_p}$ be the base-change to $\Spec(\ZZ_p)$ of the 
moduli stack of elliptic curves.  For a
locally noetherian connected scheme $S$,
and a geometric point $s$ of $S$, the $S$-points may be taken to be the
groupoid whose objects are pairs $(E,\eta)$ with:
\begin{center}
\begin{tabular}{lp{19pc}}
$E$, & an elliptic scheme over $S$, \\
$[\eta]_{GL_2(\widehat{\ZZ}^p)}$, & a $\pi_1(S,s)$-invariant
$GL_2(\widehat{\ZZ}^p)$-orbit
of linear isomorphisms: \\
& \qquad $ \eta\co V_0^{p,\infty} \xrightarrow{\cong}
V^p(E_s). $
\end{tabular}
\end{center}
The morphisms 
$$ (E, [\eta]) \rightarrow (E', [\eta']) $$
of the groupoid of $S$-points of $\mc{M}_{\mit{ell},\ZZ_p}$ are the
prime-to-$p$ quasi-isogenies of elliptic schemes over $S$ 
$$ \alpha\co E \xrightarrow{\simeq} E' $$
such that
$$
[\eta']_{GL_2(\widehat{\ZZ}^p)} = [\eta \circ \alpha_*]_{GL_2(\widehat{\ZZ}^p)}.
$$
A brief explaination of why this groupoid of elliptic schemes over $S$ 
with level
structure up to quasi-isogeny is equivalent to the groupoid of elliptic
schemes over $S$ up to isomorphism is given in \cite[Sec.~3]{Behrensbeta}.

Let $I \subset F$ be a fractional ideal of $F$.  We define $I^\vee$ to be
the fractional ideal
$$ I^\vee = \{ z \in F \: : \: \Tr_{F/\QQ}(z\bar w) \in \ZZ  \: \text{for
all $w \in I$} \} $$
We let
$$ F^* = \Hom_\QQ(F,\QQ) $$
and define a lattice
$$ I^* = \{ \alpha \in F^* \: : \: \alpha(z) \in \ZZ \: \text{for all $z
\in I$} \}. $$
The bilinear pairing
\begin{align*}
F \otimes_\QQ F & \rightarrow \QQ \\
z \otimes w & \mapsto \Tr_{F/\QQ}(z \bar w) 
\end{align*}
induces an isomorphism
$$ \alpha_I: I^\vee \rightarrow I^*. $$

\begin{lem}\label{lem:idealideal}
Let $[I] \in \Cl(F)$ be an ideal class.  Then there exists a representative
$I$ such that: 
\begin{enumerate}
\item we have $I_{(p)} = (I^\vee)_{(p)} = \mc{O}_{F,(p)} \subset F$, 
\item we have $I \subset I^\vee$.
\end{enumerate}
\end{lem}

\begin{proof}
Note that for $a \in F^\times$, $(aI)^\vee = \bar a^{-1} (I^\vee)$.  The
lemma is easily proven using the weak approximation theorem 
\cite[Thm.~6.3]{Milne}.
\end{proof}

Given an elliptic scheme $E/S$ and a fractional ideal $I \subset F$, we
associate a $2$-dimensional abelian scheme $E \otimes I/S$.  
Any chosen isomorphism
$I \to \mb Z^2$ gives rise to an isomorphism $E \otimes I \to E \times
E$.  Such an isomorphism gives a composite map
\[
\mc{ O}_F \hookrightarrow M_2(\mb Z) \to \End(E \times E) \overfrom^\sim \End(E
\otimes I)
\]
that is independent of the choice of isomorphism.  This construction
is natural in the elliptic curve and $\mc{ O}_F$-modules $I$.  In
particular, the action of $\mc{O}_F$ on $I$ gives $E\otimes I$ a canonical complex
multiplication
$$ i_I: \mc{O}_F \rightarrow \End(E \otimes I). $$
We have
$$ \Lie (E \otimes I) \otimes_{\mc{O}_{F,p}} \mc{O}_{F,u} \cong \Lie E
\otimes_{\ZZ_p} I_u. $$
In particular, Condition~\ref{eq:condition} is satisfied.

\begin{rmk}
The authors learned from the referee that the construction $E \otimes I$
has been around for some time.  See, for example, \cite[Sec.~2]{Milneav}.
\end{rmk}

The elliptic curve $E$ comes equipped with a canonical
principal polarization $\lambda\co E \to E^\vee$.  If $I$ satisfies
Lemma~\ref{lem:idealideal}(1)-(2), then one can define an
induced prime-to-$p$ polarization:
\[
\lambda_I \co E \otimes I \rightarrow E \otimes I^\vee
\xrightarrow{\lambda \otimes (\alpha_I)_*} E^\vee \otimes I^* \cong (E \otimes
I)^\vee.
\]
Note that this polarization is never principal, since the non-triviality of
the different ideal of $F$ implies 
$I^\vee \ne I$.
Replacing the ideal $I$ with the ideal $a I$ for any non-zero $a \in
\mc{O}_F$ gives an isogenous abelian variety with polarization rescaled by the positive integer
$N_{F/\QQ}(a)$.  Hence the isogeny class of
weakly polarized abelian surface $E \otimes I$ associated to $I$ depends
only on the ideal class represented by $I$.

Let $\bra{-,-}_0$ be the alternating form on $V_0 = \QQ^2$ given by
\begin{equation}\label{eq:bra0}
\bra{(x_1, x_2), (y_1, y_2)}_0 =  
\begin{bmatrix} x_1 & x_2 \end{bmatrix}
\begin{bmatrix}0 & -1 \\ 1 & 0 \end{bmatrix}
\begin{bmatrix} y_1 \\ y_2 \end{bmatrix}. 
\end{equation}
Let $\iota_0$ be the induced involution on $M_2(\QQ) = \End(V_0)$, defined
by
$$
\bra{\alpha v, w}_0 = \bra{v , \alpha^{\iota_0} w}_0. 
$$
Explicitly, we have
\begin{equation}\label{eq:iota0}
\begin{bmatrix}
a & b \\
c & d 
\end{bmatrix}^{\iota_0}
= 
\begin{bmatrix}
d & -b \\
-c & a 
\end{bmatrix}.
\end{equation}
Let $\bra{-,-}^F_0$ be the induced $\QQ$-valued 
alternating hermitian form on $V_0 \otimes F$ determined by
$$ \bra{x \otimes z, y \otimes w}_0^F = \bra{x,y}_0 \cdot \Tr_{F/\QQ}(z\bar
w). $$
There is a canonical isometry
$$
\omega: (V_0 \otimes F, \bra{-,-}^F_0)  \xrightarrow{\cong} (V, \bra{-,-})
$$
given by
$$ \omega((x_1, x_2) \otimes z) = (x_1 z, x_2 z). $$

For any $\QQ$-algebra $R$, and any $\alpha \in GL_2(R)$, we have 
$$ \alpha^{\iota_0} \alpha = \det
\alpha \in R^\times. $$
Therefore, any level structure
$$ \eta: V_0 \otimes \AF^{p,\infty} \xrightarrow{\cong} V^p(E) $$
is a similitude between $\bra{-,-}_0$ and $\bra{-,-}_{E}$, where the latter
is the Weil pairing on $V^p(E_s)$.  The inclusion $I \hookrightarrow F$
induces an isomorphism
$$ (V^p(E_s \otimes I), \bra{-,-}_{\lambda_I}) \cong (V^p(E_s)\otimes F,
\bra{-,-}_E^F). $$
The abelian variety $E_s \otimes I$ admits an induced level structure:
$$ \eta_I: V \otimes \AF^{p,\infty} \cong V_0 \otimes \AF^{p,\infty} 
\otimes F \xrightarrow{\eta \otimes 1} V^p(E_s) \otimes F \cong V^p(E_s \otimes
I). $$
Clearly, the $K_0$-orbit $[\eta_I]_{K_0}$ depends only on the
$GL_2(\widehat{\ZZ}^p)$-orbit $[\eta]_{GL_2(\widehat{\ZZ}^p)}$.

Fix a fractional ideal $I \subset F$ 
satisfying Lemma~\ref{lem:idealideal}(1)-(2).  We define a morphism of
stacks
$$ \Phi_I: \mc{M}_{\mit{ell},\ZZ_p} \rightarrow \Sh(K_0) $$
by
$$ \Phi_I(E,[\eta]) = (E \otimes I, i_I, \lambda_I, [\eta_I]). $$
Choosing representatives of each element of $\Cl(F)$ satisfying
Lemma~\ref{lem:idealideal}(1)-(2),
we get a morphism of stacks
$$
\amalg_{[I]} \Phi_I : \coprod_{\Cl(F)} \mc{ M}_{\mit{ell},\ZZ_p} \rightarrow \Sh(K_0).
$$

\begin{lem}\label{lem:aut}
Suppose $S$ is connected, let $(E,[\eta])$ be an $S$-object of
$\mc{M}_{\mit{ell},\ZZ_p}$, and let $I$ be a fractional ideal
satisfying Lemma~\ref{lem:idealideal}(1)-(2).  Then there is an
isomorphism
$$
\Aut_{\Sh(K_0)}(\Phi_I(E, [\eta])) \cong
\Aut_{\mc{M}_{\mit{ell},p}}(E, [\eta]) \times_{\{ \pm 1 \}} \mc{O}_{F}^\times.
$$
\end{lem}

\begin{rmk}
Since $F$ is a quadratic imaginary extension of $\QQ$, the group
$\mc{O}_F^\times$ is cyclic of order $4$ if $F = \QQ(i)$,
cyclic of order $6$ if $F = \QQ(\omega)$, and isomorphic to 
$\{ \pm 1 \}$ otherwise. Thus, except for two exceptions, $\Phi_I$
preserves automorphism groups.  Note that both $\QQ(i)$ and $\QQ(\omega)$
have class number $1$.
\end{rmk}

\begin{proof}[Proof of Lemma~\ref{lem:aut}]
Let 
$$ \End^0(-) = \End(-) \otimes \QQ $$ 
denote the ring of
quasi-endomorphisms.
Note that 
$$ \End^0_{\mc{O}_F}(E \otimes I) = \End^0(E) \otimes
F. $$
The $\lambda_I$-Rosati involution $\dag_I$ on $\End^0_{\mc{O}_F}(E \otimes
I)$ under this identification is given by
$$ (\alpha \otimes z)^{\dag_I} = \alpha^{\vee} \otimes \bar{z} $$
where $\alpha^\vee$ is the dual isogeny (the image of $\alpha$ under the
Rosati involution on $\End^0(E)$ corresponding to the unique weak
polarization on $E$).
An $F$-linear quasi-isogeny
$$ f: E \otimes I \rightarrow E \otimes I $$
preserves the weak polarization $\lambda_I$ if and only if $f^{\dag_I} f \in
\QQ^\times$.  If we write a general element 
$$ f = \alpha \otimes 1 + \beta \otimes \delta, $$ 
where
$\alpha, \beta \in \End^0(E)$,
we either have $\alpha = 0$, or $\beta = 0$, or both $\alpha$ and $\beta$ are
non-zero.  Assume we are in the last case.
We find the following:
\begin{align*}
\label{eq:polcompat}
f^{\dag_I} f & = (\alpha^{\vee} \otimes 1 - \beta^{\vee} \otimes \delta)
(\alpha \otimes 1 + \beta \otimes
\delta ) \\
& = (\alpha^{\vee} \alpha + N \beta^{\vee} \beta) \otimes 1 + 
(\alpha^{\vee} \beta -
\beta^{\vee} \alpha) \otimes \delta
\end{align*}
Hence the requirement is that $\alpha^{\vee} \beta$, and hence $\alpha
\beta^{-1}$ or $\beta \alpha^{-1}$, are self-dual isogenies.  
The only such endomorphisms of an elliptic curve are those which are
locally scalar multiplication, and since $S$ is connected we find
$r\alpha = s \beta$ for some integers $r,s$.  This forces the element
to be of the form $\phi \otimes z$ for some $z \in F$, and $\phi
\in \End^0(E)$. Therefore, the group of $\mc{O}_F$-linear
quasi-isogenies $E \rightarrow E$ which preserve the 
weak polarization is the group of elements 
$$ \alpha \otimes z  \in \End^0(E) \otimes F $$ \
with $\alpha \in \End^0(E)^\times$ and $z \in F^\times$.  

A quasi-isogeny of elliptic curves
$$ \alpha : E \rightarrow E $$
is prime-to-$p$ if and only if the associated
quasi-isogeny of $p$-divisible groups
$$ \alpha_*: E(p) \rightarrow E(p) $$
is an isomorphism.  An $\mc{O}_F$-linear quasi-isogeny 
$$ \beta: E \otimes I \rightarrow E \otimes I $$
is prime-to-$p$ if and only if the associated quasi-isogeny of
$p$-divisible groups
$$ \beta_* : (E \otimes I)(p) \rightarrow (E \otimes I)(p) $$
is an isomorphism.  Since there is a canonical isomorphism $(E \otimes
I)(p) \cong E(p) \otimes I$, we deduce that a quasi-isogeny $\alpha \otimes
z$ of $E \otimes I$ is prime-to-$p$, for $\alpha \in \End^0(E)$ and $z
\in F$, if and only if $\alpha$ is prime to $p$ and $z \in
\mc{O}_{F,{(p)}}^\times$.  (Here, $E(p) \otimes I$ denote the analogous
tensoring construction applied to the $p$-divisible group $E(p)$.)

Finally, suppose that $\alpha \otimes z$ preserves the
level structure $[\eta_I]_{K_0}$.  This happens if and only
if there exists a $g \in K_0$ such that the following
diagram commutes.
$$
\xymatrix{
V_0^p \otimes F 
\ar[r]^{\eta \otimes 1} 
\ar[d]_{g} 
& V^p(E_s) \otimes F 
\ar[d]^{\alpha_* \otimes z} 
\\
V_0^p \otimes F 
\ar[r]_{\eta \otimes 1}
& V^p(E_s) \otimes F
}
$$
This will happen if and only if
$$ \eta^{-1} \alpha_* \eta \otimes z \in K_0 $$
which in turn happens if and only if for the lattice
$$ L_0 = \ZZ^2 \subset \QQ^2 = V_0 $$
we have
\begin{align*}
\alpha_*(\eta(\widehat{L}_0^p)) & = \eta(\widehat{L}_0^p), 
\\
z & \in \widehat{\mc{O}}^p_F.
\end{align*}
The first condition is equivalent to asserting that the
quasi-isogeny $\alpha$ preserves the level structure
$[\eta]_{GL_2(\widehat{\ZZ}^p)}$.

Putting this all together,
we conclude that $\alpha \otimes z$ represents an
automorphism of $\Phi_I(E, [\eta])$
in $\Sh(K_0)$ if and only if $\alpha$ is an automorphism of $(E, [\eta])$
in $\mc{M}_{\mit{ell},\ZZ_p}$ and $z \in \mc{O}_F^\times$.  The lemma follows
from the fact that 
$$ (-\alpha) \otimes z = \alpha \otimes (-z). $$
\end{proof}

\begin{lem}\label{lem:isoclass}
Suppose that $I$ and $I'$ are two ideals satisfying
Lemma~\ref{lem:idealideal}(1)-(2), let $(E, [\eta])$ and $(E', [\eta'])$
be objects of $\mc{M}_{\mit{ell},\ZZ_p}(S)$.  Then $\Phi_I(E)$ is
isomorphic to $\Phi_{I'}(E')$ in $\Sh(K_0)(S)$ if and only if 
$(E,[\eta])$ is
isomorphic to $(E', [\eta'])$ in $\mc{M}_{\mit{ell},\ZZ_p}(S)$ 
and $[I] = [I'] \in \Cl(F)$.
\end{lem}

\begin{proof}
Clearly if $(E, [\eta])$ is isomorphic to $(E', [\eta'])$ then $\Phi_I(E,
[\eta])$ is isomorphic to $\Phi_I(E', [\eta'])$.  Moreover, if $[I'] =
[I]$, then there is an $a \in F^\times$ such that $I' = aI$.  Since
$I_{(p)} = I'_{(p)}$, we deduce that $a \in \mc{O}_{F,{(p)}}$.  The
mapping
$$ 1 \otimes a: E \otimes I \rightarrow E \otimes I' $$
is a prime-to-$p$, $\mc{O}_F$-linear, quasi-isogeny that preserves the
weak polarization and level structure.

Conversely, suppose that
$$ f: E \otimes I \rightarrow E' \otimes I' $$ 
gives an isomorphism between $\Phi_I(E)$ and $\Phi_{I'}(E')$.  There is an
isomorphism
\begin{equation}\label{eq:homiso}
\Hom^0_{\mc{O}_F}(E \otimes I, E' \otimes I') \cong \Hom^0(E,E')
\otimes F.
\end{equation}
The pair of polarizations $\lambda_I$, $\lambda_{I'}$ induces a
homomorphism
$$ \dag_{I,I'}: \Hom^0_{\mc{O}_F}(E \otimes I, E' \otimes I')
\rightarrow 
\Hom^0_{\mc{O}_F}(E' \otimes I', E \otimes I) $$
such that an $\mc{O}_F$-linear quasi-isogeny $g$ preserves the weak 
polarization if and only if
$$ g^{\dag_{I,I'}} \circ g \in \QQ^\times. $$
Under the isomorphism (\ref{eq:homiso}), we have
$$ (\beta \otimes w)^{\dag_{I,I'}} = \beta^\vee \otimes \bar{w}. $$
Thus similar arguments as given in the proof of Lemma~\ref{lem:aut}
imply that there exists a prime-to-$p$ isogeny
$$ \alpha: E \rightarrow E' $$
and $z \in \mc{O}_{F,(p)}^\times$ 
so that $f = \alpha \otimes z$.  Since $f$ preserves level
structures, we deduce that:
\begin{align*}
[\alpha_* \eta] & = [\eta'],\\
zI_{\ell} & = I'_\ell, \quad \text{for all $\ell \ne p$.}
\end{align*}
Since $I_p = I'_p$, we conclude that $I' = zI$.
\end{proof}

The following theorem gives a complete description of $\Sh(K_0)$ in terms
of the moduli stack of elliptic curves.

\begin{thm}\label{thm:etale}
$\quad$
\begin{enumerate}
\item If $F = \mb Q(i)$, the map $\mc{ M}_{\mit{ell},\ZZ_p} \to
  \Sh(K_0)$ is a degree 2 Galois cover of Deligne-Mumford stacks.
\item If $F = \mb Q(\omega)$, the map $\mc{ M}_{\mit{ell},\ZZ_p} \to
  \Sh(K_0)$ is a degree 3 Galois cover of Deligne-Mumford stacks.
\item In all other cases, the map $\coprod_{\Cl(F)} \mc{ M}_{\mit{ell},\ZZ_p} 
\to
  \Sh(K_0)$ is an equivalence.
\end{enumerate}
\end{thm}

We remark that in the first two cases the Galois group acts trivially
on the underlying coarse moduli object, but acts via nontrivial
automorphisms on the points of the stack.  A generic point of the
stack $\mc{ M}_{\mit{ell},\ZZ_p}$ has automorphism group $\mb Z/2
\cong \mb Z^\times$, whereas a generic point of $\Sh(K_0)$ associated
to the quadratic imaginary field $F$ has automorphism group ${\mathcal
  O}_F^\times$.

The remainder of this section will be devoted to proving
Theorem~\ref{thm:etale}.  We will first establish the following weaker
version of Theorem~\ref{thm:etale}.

\begin{lem}\label{lem:proper}
$\quad$
\begin{enumerate}
\item If $F = \mb Q(i)$, the map $\mc{ M}_{\mit{ell},\ZZ_p} \to
  \Sh(K_0)$ is a degree 2 Galois cover of a connected component.
\item If $F = \mb Q(\omega)$, the map $\mc{ M}_{\mit{ell},\ZZ_p} \to
  \Sh(K_0)$ is a degree 3 Galois cover of a connected component.
\item In all other cases, the map $\coprod_{\Cl(F)} \mc{
M}_{\mit{ell},\ZZ_p} \to
  \Sh(K_0)$ is an inclusion of a set of connected components.
\end{enumerate}
\end{lem}

\begin{proof}
We know that these maps are \'etale, as both moduli are locally
universal deformations of the associated $p$-divisible groups.  Given
Lemmas~\ref{lem:aut} and \ref{lem:isoclass}, to
finish the justification of this statement, we must prove that these
maps are proper.

Suppose that we are given a discrete valuation ring $R$ over $(\mc{
  O}_F)_{(p)}$ with fraction field $K$, an elliptic curve $E$ over
$K$, and an extension of $E \otimes I$ to a polarized abelian variety
$A$ over $R$ with $\mc{ O}_F$-action serving as an $R$-point of the
moduli.  Then $A$ is a N\'eron model for $E \otimes I$, and in
particular the universal mapping property allows the direct product
decomposition $E \otimes I \cong E \times E$ to extend uniquely to
$A$, together with the $\mc{ O}_F$-action.  The same holds true for
$A^\vee$ and the polarization.
\end{proof}

We are left with showing that the map
$$ \Phi: \coprod_{\Cl(F)} \mc{ M}_{\mit{ell},\ZZ_p} \rightarrow
  \Sh(K_0) $$
is surjective on $\pi_0$.  Assume this is not true.  Then there is
a connected component $Y$ of $\Sh(K_0)$ that is disjoint from the image of
$\Phi$.  Since $\Sh(K_0)$ possesses an \'etale cover by a quasi-projective
scheme over $\Spec(\ZZ_p)$, it follows that $Y$ must have an
$\bar{\FF}_p$-point $y_0$.  By Serre-Tate theory, there exists a lift of
this point to a $\QQ_p^{nr}$-point $y$.
Choosing an isomorphism $\bar{\QQ}_p \cong \CC$, we
see that $y$ corresponds to a $\CC$-point of $\Sh(K_0)$ which is not in the
image of $\Phi$.  To arrive at a contradiction, and hence prove
Theorem~\ref{thm:etale}, it suffices to demonstrate that $\Phi$ is
surjective on $\CC$-points.  

Lemma~\ref{lem:proper} implies that $\Phi$
surjects onto a set of connected components of $\Sh(K_0)_\CC$.  Therefore
we simply must prove that the induced map
\begin{equation}\label{eq:pi0}
\Phi_*: \pi_0(\coprod_{\Cl(F)} \mc{M}_{\mit{ell},\CC}) \rightarrow
\pi_0(\Sh(K_0)_\CC) 
\end{equation}
is an isomorphism.  Since $\mc{M}_{\mit{ell},\CC}$ is connected, 
the left-hand-side of (\ref{eq:pi0}) is isomorphic to
$$ \Cl(F) = F^\times \setminus (\AF^\infty_F)^\times /
\widehat{\mc{O}}_F^\times $$
whereas
Theorem~9.3.5 and Remark~9.3.6 of \cite{taf} shows that the right-hand-side
is isomorphic to 
$$ GU(\QQ) \setminus GU(\AF^\infty) / K_0. $$
The map
\begin{equation}\label{eq:pi0Phi}
\pi_0\Phi: F^\times \setminus (\AF^\infty_F)^\times /
\widehat{\mc{O}}_F^\times \rightarrow
GU(\QQ) \setminus GU(\AF^\infty) / K_0
\end{equation}
induced by $\Phi$ under these isomorphisms is the map of adelic quotients
induced by the inclusion of the center $\mr{Res}_{F/\QQ} \GG_m$ of $GU$.
This can be seen as follows: given a fractional ideal $I \subset F$,
picking a generator $z_{\mf{p}}$ of $I_{\mf{p}}$ 
for each prime $\mf{p}$ of $F$, we get an element $(a_{\mf{p}}) \in
\AF_F^\infty$.  Given a complex elliptic curve $C = \CC/\Lambda$
associated to a lattice $\Lambda$ with basis $(e_1, e_2)$, 
the abelian variety $\Phi(C)$ is given
by $\Phi(C) = \CC^2/\Lambda'$ where $\Lambda'_{\mf{p}}$ has basis
$(a_{\mf{p}}e_1, a_{\mf{p}}e_2)$.

\begin{lem}
The map $\pi_0 \Phi$ of (\ref{eq:pi0Phi}) is an isomorphism.
\end{lem}

\begin{proof}
The map is easily seen to be a monomorphism.  Thus it suffices to show that
$$
h(GU) := \vert GU(\QQ) \setminus GU(\AF^\infty) / K_0 \vert = h(F) 
$$
where $h(F)$ is the class number of $F$.  Shimura \cite[Thm.~5.24(ii)]{Shimura}
computed the class numbers of indefinite unitary similitude groups.  In our
particular case his formula gives
$$ h(GU) = h(\QQ)[\mf{C}: \mf{C}_0][\mf{E}, f(\mc{O}_F^\times)]. $$
Here $h(\QQ) = 1$ is the class number of $\QQ$.  The group $\mf{C}$ is
the class
group of $F$, and $\mf{C}_0$ denotes the subgroup generated by ideals which
are invariant under conjugation.  However, in
\cite[Thm.~5.24(i)]{Shimura}, Shimura argues that $[\mf{C}: \mf{C}_0]$ is
the class number for a unitary group associated to a hermitian form on an
odd dimensional $F$-vector space.  Specializing this result to the 
$1$-dimensional case, we deduce
$$ [\mf{C}:\mf{C}_0] = h(T) $$
where $h(T)$ is the class number of the torus
$$ T = \ker(N_{F/\QQ}: \mr{Res}_{F/\QQ} \GG_m \rightarrow \GG_m). $$
Just prior to \cite[Thm.~5.24]{Shimura}, the group $\mf{E}$ is defined, and
shown to be isomorphic to the product $(\ZZ/2)^v$, where $v$ denotes the
number of primes $\ell$ which ramify in $F$ for
which the corresponding local hermitian form has non-trivial anisotropic
subspace.  Since our hermitian form is actually isotropic, the 
the number $v$ is equal to $u$, the number of primes which ramify in $F$.
The map $f$ maps $\mc{O}_F^\times$ into $\mf{E}$ with kernel
$(\mc{O}_F^\times)^2$.  We therefore deduce that
$$ [\mf{E}:f(\mc{O}_F^\times)] = 2. $$
The class number $h(T)$ is given by
$$ h(T) = h(F)/2^{u-1}, $$
(see, for instance, p.375, Equation~(16) of
\cite{Shyr}) and this give the desired result.
\end{proof}

\section{Computation of $\TAF_{GU}(K_0)$}
\label{sec:elliptictensor}

In this section we continue to take $V$, $\bra{-,-}$, $GU$, and $K_0$ as in
Section~\ref{sec:tensor}.  In this section we give a complete description of the spectrum
$\TAF_{GU}$ for this choice of initial data.

Let $(\mbf{E}, [\pmb{\eta}])$ be the universal elliptic curve over
$\mc{M}_{\mit{ell}, \ZZ_p}$.  Let 
$$ \mc{E}_{GL_2} = \mc{E}_{\mbf{E}(p)} $$
be the sheaf of $E_\infty$-ring spectra 
associated to the $p$-divisible group $\mbf{E}(p)$.  The global sections
give the $p$-completion of the spectrum of topological modular forms:
$$ \TMF_p = \mc{E}_{GL_2}((\mc{M}_{\mit{ell}})^\wedge_p). $$
The pullback of the $p$-divisible group $\mbf{A}(u)$ associated to the
universal abelian scheme over $\Sh(K_0)$ under the map
$$ \Phi: \coprod_{\Cl(F)} \mc{M}_{\mit{ell}, \ZZ_p} \rightarrow \Sh(K_0)
$$
is given by
$$ \Phi^* \mbf{A}(u) \cong \mbf{E}(p). $$
We deduce that there is an isomorphism of presheaves
\begin{equation}\label{eq:presheafiso}
\nabla^* \mc{E}_{GL_2} \cong \Phi^* \mc{E}_{GU}
\end{equation}
where 
$$ \nabla: \amalg_{\Cl(F)} \mc{M}_{\mit{ell}, \ZZ_p} \rightarrow
\mc{M}_{\mit{ell}, \ZZ_p}
$$
is the codiagonal.   

\begin{thm}\label{thm:tensor}
We have the following equivalences of $E_\infty$-ring spectra.
\begin{enumerate}
\item If $F = \mb Q(i)$, there is an equivalence
$$ \TAF_{GU}(K_0) \simeq \TMF_p^{hC_2}. $$
\item If $F = \mb Q(\omega)$, there is an equivalence
$$ \TAF_{GU}(K_0) \simeq \TMF_p^{hC_3}. $$
\item In all other cases, we have
$$ \TAF_{GU}(K_0) \cong \prod_{\Cl(F)} \TMF_p. $$
\end{enumerate}
\end{thm}

\begin{proof}
(3) follows immediately from applying the global sections functor to
Equation~\ref{eq:presheafiso}.  (1) and (2) are established by noting that
since the presheaf $\mc{E}_{GU}$ is Jardine fibrant, it satisfies descent
with respect to the Galois cover $\Phi$.
\end{proof}

We pause to give a precise description of the group actions of
Theorem~\ref{thm:tensor} on the spectrum $\TMF_p$.  In
\cite[1.2.1]{markmodular}, the first author described certain operations
$$ \psi^*_{[k]}: \TMF_p \rightarrow \TMF_p $$
for $k$ coprime to $p$ which are analogs of the Adams operations on
$K$-theory.  These operations give an action of $\ZZ_{(p)}^\times$ on
$\TMF_p$.
The functoriality of the sheaf $\mc{E}_{\GG}$ with respect to the
$p$-divisible group $\GG$ implies that the central action of
$\ZZ_p^\times$ on the $p$-divisible group $\GG$ by isomorphisms induces an
extension of the $\ZZ_{(p)}^\times$-action on $\TMF_p$ to $\ZZ_p^\times$. 
The action factors through $\ZZ_p^\times/\{ \pm 1 \}$, since for any
elliptic curve $E$, the isogeny
$$ [-1]: E \rightarrow E $$
is an isomorphism.  Since $p$ is assumed to split in $F$, there is an
inclusion 
$$ \mc{O}_F^\times \hookrightarrow \mc{O}_{F,u}^\times \cong \ZZ_p^\times $$
and hence an action of $\mc{O}_F^\times/\{ \pm 1\}$ on $\TMF_p$.

\begin{cor}
\label{cor:tensor}
The homotopy groups of the 
spectra of topological automorphic forms $\TAF_{GU}(K_0)$ are computed as
follows.
\begin{enumerate}
\item If $F = \mb Q(i)$, there is an isomorphism
$$ \TAF_{GU}(K_0)_* \cong 
\mb Z_p[c_4,c_6^2,\Delta^{-1}]^\wedge_p \subset \ZZ_p[c_4, c_6,
\Delta^{-1}]^\wedge_p. $$
\item If $F = \mb Q(\omega)$, 
there is an isomorphism
$$ \TAF_{GU}(K_0)_* \cong \mb Z_p[c_4^3,c_6,\Delta^{-1}]^\wedge_p \subset 
\ZZ_p[c_4, c_6, \Delta^{-1}]^\wedge_p. $$
\item In all other cases, there is an isomorphism
$$ \TAF_{GU}(K_0)_* \cong \prod_{\Cl(F)} \pi_* \TMF_p. $$
\end{enumerate}
Here, $c_4$, $c_6$, and $\Delta$ are the standard integral modular forms.
\end{cor}

\begin{proof}
Case (3) follows immediately from Theorem~\ref{thm:tensor}.  Suppose we are
in cases (1) or (2), and that $G$ is the Galois group of the cover $\Phi$,
so that $G$ is either $C_2$ or $C_3$, respectively.
Note that since $p$ is assumed to split in $F$, cases (1) and (2) of
Theorem~\ref{thm:tensor} only
occur when the prime $p$ does not divide $6$.  Therefore, the
associated homotopy fixed point spectral sequence 
$$ H^s(G, \pi_t \TMF_p) \Rightarrow \pi_{t-s} \TAF_{GU}(K_0) $$
collapses to the $0$-line.  
Since $p$ does not divide $6$, we have
$$ \pi_* \TMF_p \cong \ZZ_p[c_4, c_6, \Delta^{-1}]^\wedge_p. $$
The proof of the corollary is completed by identifying the action of the
group $G$.
The $p$-divisible summand $A(u)$ is $E(p) \otimes_{\mb Z_p}
\comp{I}_u$, with $\mc{ O}_F$ acting via its image in $\comp{(\mc{
O}_F)}_u = \mb Z_p$.  Therefore, the roots of unity in $\mc{
O}_F^\times$ act on the $p$-divisible group (and hence on invariant 
1-forms in the formal part) by multiplication by roots of
unity in $\mb Z_p^\times$.  As a result, the action on forms of weight
$k$ is through the $k$'th power map.  Thus, in case (1), $i$ acts trivially on 
$c_4$ and by negation on $c_6$; in case (2), $\omega$
acts trivially on $c_6$ and by multiplication by $\omega$ on $c_4$.
\end{proof}

\begin{rmk}
The existence of these summands of $\TMF$ can be derived directly.
The previously described action of $\mb Z_p^\times$ on the
$p$-divisible group of $\mc{ M}_{ell,\mb Z_p}$ gives rise to an
action by a group of $(p-1)$'st roots of unity $\mu_{p-1} \subset
\mb Z_p^\times$ via scalar multiplication, and hence this group acts
on the associated spectrum $\TMF_p$.  The invariants under this
action form a summand analogous to the Adams summand, generalizing
the summands for $\mb Q(i)$ and $\mb Q(\omega)$.
\end{rmk}

\section{A quotient construction}\label{sec:quotient}

Let $\mf{d} \subset F$ be the different of $F$,  
let  
$L'$ be the $\mc{O}_F$-lattice 
$$ L' = \mc{O}_F \oplus \mf{d}^{-1} \subset F^2 = V. $$
Let 
$K_1$ denote
the compact open subgroup of $GU(\AF^{p,\infty})$ given by
$$ K_1 = \{ g \in GU(\AF^{p,\infty}) \: : \: g(\widehat{L'}^p) =
\widehat{L'}^p \}. $$
The significance of the lattice $L'$ is that, unlike the lattice $L$ of
Section~\ref{sec:tensor}, the lattice $L'$ is self-dual with respect to
$\bra{-,-}$, in the sense that we have
$$ L' = \{ x \in V \: : \: \bra{x, L'} \subseteq \ZZ \}. $$
This implies that the associated Shimura stack $\Sh(K_1)$ admits a moduli
interpretation
where all of the points are represented by principally polarized abelian
schemes with complex multiplication by $F$ (see Remark~\ref{rmk:principal}).  
This should be contrasted with
the moduli interpretation of $\Sh(K_0)$ developed in
Section~\ref{sec:tensor}, where none of the polarized abelian schemes $(E
\otimes I, \lambda_I)$ were principally polarized.

In this section we will identify one connected component of the 
associated Shimura variety $\Sh(K_1)$
with the 
$p$-completion of a quotient of the moduli stack of elliptic curves with
$\Gamma_0(N)$-structure.  Here, as always in this paper, $F = \QQ(-N)$,
where $N$ is a positive square-free integer relatively prime to $p$.

We recall that any elliptic curve $E$ has a canonical principal polarization
$\lambda$, and each isogeny $f$ of elliptic curves has a dual
$f^\vee = \lambda^{-1} f^\vee \lambda$.  The composite $f^\vee f$ is
multiplication by the degree of the isogeny.

In \cite[Sec.~6.4]{taf}, the authors observed that the moduli
interpretation of $\Sh(K_1)$ as a moduli stack of polarized $F$-linear abelian
schemes up to \emph{isogeny} with level structure could be replaced with a 
moduli
interpretation as a moduli stack of polarized abelian schemes up to
\emph{isomorphism} without level structure.  

Specifically,  for a
locally noetherian connected $\ZZ_p$-scheme $S$,
the $S$-points of $\Sh(K_1)$ is the groupoid whose objects are tuples
$(A,i,\lambda)$, with:
\begin{center}
\begin{tabular}{lp{19pc}}
$A$, & an abelian scheme over $S$ of dimension $2$,\\
$\lambda\co A \rightarrow A^\vee$, & a $\ZZ_{(p)}$-polarization, \\
$i\co \mc{O}_{F,(p)} \hookrightarrow \End(A)_{(p)}$, & an inclusion of
rings, such that the $\lambda$-Rosati involution is compatible with
conjugation.
\end{tabular}
\end{center}
subject to the following two conditions:
\begin{enumerate}
\item the coherent sheaf of $\mc{O}_S$-modules 
$\Lie A \otimes_{\mc{O}_{F,p}} \mc{O}_{F,u}$ is
locally free of rank $1$,

\item for a geometric point $s$ of $S$, there exists 
a $\pi_1(S,s)$-invariant
$\mc{O}_F$-linear similitude: 
$$ \eta\co (\widehat{L'}^p, \bra{-,-}) \xrightarrow{\cong}
(T^p(A_s), \bra{-,-}_\lambda). $$
(Here, $T^p(A_s)$ is the Tate module of $A_s$ away from $p$.)
\end{enumerate}

The morphisms 
$$ (A, i, \lambda) \rightarrow (A', i', \lambda') $$
of the groupoid of $S$-points of $\Sh(K_1)$ are isomorphisms 
of abelian schemes 
$$ \alpha\co A \xrightarrow{\simeq} A' $$
such that
\begin{alignat*}{2}
\lambda & = r\alpha^\vee \lambda' \alpha, 
\quad && 
r \in \ZZ_{(p)}^\times, \\
i'(z)\alpha & = \alpha i(z), 
\quad &&
z \in \mc{O}_{F,(p)}.
\end{alignat*}

\begin{rmk}\label{rmk:principal}
Observe that the tuple $(A,i,\lambda)$ only depends on the weak
polarization class of $\lambda$.  There is a unique similitude class of
$\widehat{\mc{O}}^p_F$-lattice in $V^{p,\infty}$ which is self-dual.
Therefore, Condition~(2) above is \emph{equivalent} to the condition that
the weak polarization class of $\lambda$ contains a representative which
is principal.  We may and will restrict ourselves to principal
polarizations in this section.
\end{rmk}

Let $\mc{ M}_0(N)$ denote the moduli stack (over $\ZZ[1/N]$) whose
$S$-points are the groupoid whose objects are pairs $(E,H)$ where 
$E$ is a (nonsingular) elliptic scheme over $S$,  and $H \le E$ is a
$\Gamma_0(N)$-structure, i.e. a cyclic subgroup of order $N$.  The
morphisms of the groupoid of $S$-points consist of isomorphisms of elliptic
curves which preserve the level structure.
Note that we have $\mc{M}_{\mit{ell}} = \mc{M}_0(1)$.  We may interpret the
$S$-points
of this moduli as being isogenies $q\co E \to \bar E$ of elliptic
curves whose kernel is cyclic of order $N$.

We will construct a morphism
\begin{align*}
\Phi': \mc{M}_0(N)_{\ZZ_p} & \rightarrow \Sh(K_1), \\
(E, H) & \mapsto (\Phi'(E), i_{E}, \lambda_E).
\end{align*}
We break the construction down into two cases.

\subsection*{Case I: $-N \equiv 2,3 \mod 4$}

In this case, the ring of integers is given by
$$ \mc{ O}_F \cong \mb Z[x]/(x^2 +
N). $$
We define our abelian scheme to be 
$$ \Phi(E) := E \times \bar E, $$
with polarization
the component-wise principal polarization 
$$ \lambda_E = \lambda \times \lambda. $$  
We
define a complex multiplication 
$$ i_E: \mc{ O}_F \to \End(E \times \bar E) $$ 
by the map
\[
x \mapsto \tau \circ (q,-q^\vee).
\]
Here, $\tau$ is the twist map $\bar{E} \times E \to E \times \bar{E}$.  
The dual
of this element is
\[
x^\vee = (q^\vee, -q) \circ \tau = \tau \circ (-q,q^\vee) = -x,
\]
so the Rosati involution induces complex conjugation on $\mc{ O}_F$.

As $p$ splits in $F$, let $a \in \mb Z_p^\times$ be the image of $x$
corresponding to the prime $u$, satisfying $a^2 + N = 0$.  The
canonical rank $1$ summand of the coherent sheaf $\Lie (E \times \bar{E})$ 
is the
image of $\Lie E$ under the map
\[
\left(1 \times \frac{q}{a}\right) \circ \Delta\co \Lie E \to \Lie (E \times
\bar E).
\]

\subsection*{Case II: $-N \equiv 1 \mod 4$}

In this case, the ring of integers $\mc{ O}_F$ is $\mb Z[y]/(y^2 + y
+ \frac{N+1}{4})$.  We would like to define our abelian variety as in
Case I; however, this definition would not allow an
action of the full ring of integers $\mc{ O}_F$.

Instead, we take 
$$ \Phi'(E) = \frac{E \times \bar{E}}{E[2]} $$
where we have taken quotients by the image of the composite
$$ E[2] \xrightarrow{\Delta} E[2] \times E[2] \hookrightarrow E \times
E \xrightarrow{1 \times q} E \times \bar{E}. $$
(Note that in this case, the kernel of $q$ has order prime to $2$.)
The polarization
$$ 2\lambda \times 2\lambda : E \times \bar{E} \rightarrow E^\vee \times
\bar{E}^\vee $$
descends to the quotient to give a principal polarization
$$ \lambda_E: \Phi'(E) \rightarrow \Phi'(E)^\vee. $$
The action of the order 
$$ \ZZ + \ZZ(2y) \subset F $$
on $E \times \bar{E}$ given in Case I 
descends to the quotient to give complex multiplication
$$ i_E: \mc{O}_F \rightarrow \End(\Phi'(E)) $$
over the resulting quotient.

We note that the endomorphism
$$ \begin{bmatrix}2&0\\-q&1\end{bmatrix} \in \End(E \times \bar{E}) $$
factors through the quotient $\Phi'(E)$ to give an isomorphism
$$ \Phi'(E) \xrightarrow{\cong} E \times \bar{E}. $$
We will use this to display explicit formulas.

On $E \times \bar E$, the induced polarization is defined by
\[
\lambda_E = (\lambda,\lambda) \circ A = (\lambda,\lambda) \circ 
\begin{bmatrix}
\frac{1+N}{2} & q^\vee \\
q & 2
\end{bmatrix}\co E \times \bar E \to E^\vee \times \bar E^\vee
\]
The matrix $A$ is symmetric with respect to transpose dual $\dag$, 
and is positive definite, as required to define a polarization.

We define complex multiplication 
$$ i_E: \mc{ O}_F \to \End(E \times \bar E) $$ 
by
\[
y \mapsto 
\begin{bmatrix}
\frac{-N-1}{2} & q^\vee \\
\frac{N+1}{4}q & \frac{N-1}{2}
\end{bmatrix}.
\]
This endomorphism satisfies the equation $y^2 + y + \frac{N+1}{4}$
and conjugate-commutes with the polarization.  This last follows from
the identity $y^\dag A = A(-1-y)$, or
\[
\begin{bmatrix}
\frac{-N-1}{2} & \frac{N+1}{4}q^\vee \\
q & \frac{N-1}{2}
\end{bmatrix}
\begin{bmatrix}
\frac{1+N}{2} & q^\vee \\
q & 2
\end{bmatrix}
=
\begin{bmatrix}
\frac{1+N}{2} & q^\vee \\
q & 2
\end{bmatrix}
\begin{bmatrix}
\frac{N-1}{2} & -q^\vee \\
\frac{-N-1}{4}q & \frac{-N-1}{2}
\end{bmatrix}.
\]

As in the previous case, one can check that the summand of $\Lie \Phi'(E)$
is canonically isomorphic to the rank 1 summand of $\Lie (E
\times \bar E)$.  

\subsection*{Isomorphisms of objects}
We now consider when two such abelian varieties $(E \times \bar E)$
and $(E' \times \bar E')$ can become isomorphic in the moduli.  This
proceeds in a way similar to Section~\ref{sec:tensor}.

We note that after rationalizing $\Hom$-sets, both cases become
isomorphic to the abelian variety $E \times \bar E$ as in Case~I.  The
rationalized set of maps $E \times \bar E \to E' \times \bar E'$ is
the set of matrices
\[
\left\{\begin{bmatrix}
\alpha & \beta q^\vee \\
q \gamma & q \delta q^\vee
\end{bmatrix}\ \Big | \ \alpha, \beta, \gamma, \delta \in \Hom(E,E')
\otimes \mb Q\right\}.
\]
The set of such elements that commute with
$\begin{bmatrix}0& -q^\vee \\q & 0\end{bmatrix}$ is
\[
\left\{\begin{bmatrix}
\alpha & \beta q^\vee \\
-q \beta & \frac{1}{N}q \alpha q^\vee
\end{bmatrix}\ \Big | \ \alpha, \beta \in \Hom(E,E')
\otimes \mb Q\right\}.
\]
The set of such elements $f$ that preserve the polarization, i.e. such
that $f^\dag f$ is scalar, are the elements satisfying $\alpha^\vee
\beta = \beta^\vee \alpha$, or $\alpha^\vee \beta$ is symmetric.  As
in Section \ref{sec:tensor}, the only symmetric endomorphisms
of an elliptic curve are scalar.  We find that there is an isogeny
$\phi$ such that $\alpha = a\phi$ and $\beta = b\phi$ for some
endomorphism $\phi$ and some scalars $a,b$.  In order for $f$ to
additionally be an {\em isomorphism}, we must have $f^\vee f = 1$,
implying $\alpha^\vee \alpha + (q\beta)^\vee (q\beta) = 1$, and hence
\[
\deg(\alpha) + \deg(q\beta) = 
\deg(a\phi) + N \deg(b\phi) = 1.
\]

In Case I, as the $\Hom$-set embeds into its rationalization, we must
have such a $2\times2$ matrix whose entries are genuine homomorphisms.
As $g^\vee g$ is the degree of the isogeny $g$, this can only
occur if one of $\alpha, q\beta$ is an isomorphism and the other is
zero.

In Case II, we note that since the abelian variety differs from $E
\times \bar E$ in this expression by a subgroup of $2$-torsion, the
entries of the matrix may not be genuine homomorphisms, but
multiplying any of them by $2$ is.  Therefore, we find that $2\alpha$
and $2q\beta$ are homomorphisms with $\deg(2\alpha) + \deg(2q\beta)
= 4$.  There are only the following
possibilities:
\begin{enumerate}
\item $\deg(2\alpha) = 4$, $\deg(q\beta) = 0$.  Such elements come from
  isomorphisms $E \to E'$ which respect the level structure.
\item $\deg(\alpha) = 0$, $\deg(2q\beta) = 4$.  Such elements are
  compositions of isomorphisms $E \to E'$ with the ``twist'' map
  $(1,-1) \circ \Delta\co E \times \bar E \to \bar E \times E$.  
\item $\deg(2\alpha) = 2$, $\deg(2q\beta) = 2$.  This would force $N =
  1$, which is not in case II.
\item $\deg(2\alpha) = 1$, $\deg(2q\beta) = 3$.  This forces $N = 3$,
  and such elements are compositions of isomorphisms $E \to E'$ with
  the matrix
  $\begin{bmatrix}-\frac{1}{2}&\frac{1}{2}q^\vee\\-\frac{1}{2}q&-\frac{1}{2}\end{bmatrix}$
  from $E \times \bar E$ to itself.  This element is precisely a third
  root of unity from $\mc{ O}_F^\times$.
\item $\deg(2\alpha) = 3$, $\deg(2q\beta) = 1$.  This forces $N = 3$,
  and such elements are compositions of isomorphisms $E \to E'$, the
  twist map $E \times \bar E \to \bar E \times E$, and matrices of
  the previous type.
\end{enumerate}

Let
$$ w: \mc{M}_0(N) \rightarrow \mc{M}_0(N) $$
be the Fricke involution, which on $S$-points is given by
$$ (E, q) \mapsto (\bar{E}, q^\vee). $$
Clearly, $w^2 = \mr{Id}$.  Let $\mc{M}_0(N) \mmod \bra{w}$ denote the stack
quotient by the action of $w$.

Observe that the map
$$ 
\tau' = 
\begin{bmatrix}
0 & -1 \\
1 & 0
\end{bmatrix}
: E \times \bar{E} \xrightarrow{\cong} \bar{E} \times E. $$
induces an isomorphism between the points $\Phi'(E,q)$ and $\Phi'(\bar{E},
q^\vee)$ of $\Sh(K_1)$.
Therefore, the map $\Phi'$ factors through the quotient by the Fricke
involution to give a map
$$ \Phi': \mc{M}_0(N)_{\ZZ_p}\mmod\bra{w} \rightarrow \Sh(K_1). $$

\begin{rmk}\label{rmk:Q(i)}
One must treat the case where $N = 1$ as exceptional, where 
$$ E \times \bar{E} = E \times E $$
has complex multiplication by $\mc{O}_F = \ZZ[i]$, and the isomorphism $\tau'$
corresponds to the action by $i$.  The Fricke involution is the
identity on the moduli, because there is no level structure, but the
rescaled involution on the $p$-divisible group is nontrivial.  The
compact open subgroup
$K_1$ is actually conjugate to $K_0$ in this case, so there is an
isomorphism of Deligne-Mumford stacks
$$ \Sh(K_1) \cong \Sh(K_0). $$
(This is the only case where the subgroups $K_0$ and $K_1$ are conjugate.)
\end{rmk}

\begin{rmk}
The natural $1$-dimensional summand of the $p$-divisible group, and
similarly the Lie algebra, of $\Phi'(E)$ are identified with the
$p$-divisible group and Lie algebra of $E$.  Under this
identification, the Lie algebra of $w(E,q) = (\bar E, q^\vee)$ is
identified with that of $E$ via a rescaling of the isogeny
\[
q\co \Lie E \to \Lie \bar E
\]
by dividing by $a = \sqrt{-N} \in \mb Z_p$.
\end{rmk}

\begin{thm}\label{thm:quotient}
$\quad$
\begin{enumerate}
\item If $F = \mb Q(i)$, then the  
map 
$$ \Phi': \mc{M}_0(1)_{\ZZ_p} = \mc{ M}_{\mit{ell},\ZZ_p} \to
  \Sh(K_1) $$ 
is a degree 2 Galois cover.
\item If $F = \mb Q(\omega)$, the map 
$$ \Phi': \mc{ M}_0(3)_{\ZZ_p} \mmod \bra{w} \to
  \Sh(K_1) $$ 
is a degree 3 Galois cover onto a connected component.
\item In all other cases, the map 
$$ \Phi': \mc{M}_0(N)_{\ZZ_p} \mmod \bra{w} \to
  \Sh(K_1) $$ 
is an inclusion of a connected component.
\end{enumerate}
\end{thm}

\begin{proof}
Comparing the morphisms in the induced map on groupoids of $S$-points, we see 
from our Case I analysis that an isomorphism
$$ E \times \bar{E} \rightarrow E' \times \bar E' $$
is either of the form
$$ \alpha \times \frac{1}{N} q' \alpha q^\vee $$
for an isomorphism 
$$ \alpha: (E,q) \xrightarrow{\cong} (E',q') $$
in $\mc{M}_0(N)(S)$
or of the form
$$ (\alpha \times \frac{1}{N} q' \alpha q) \circ \tau' $$
for an isomorphism 
$$ \alpha: (\bar{E}, q^\vee) \xrightarrow{\cong} (E', q') $$
in $\mc{M}_0(N)(S)$.  (The case $F = \QQ(i)$ is an exception, as noted in
Remark~\ref{rmk:Q(i)}.) 

The Case II situation is analogous, since the
isomorphisms above descend through the diagonal quotient by the $2$-torsion
of the elliptic curve.  The only exception is the case where $F =
\QQ(\omega)$, where, as noted in the Case II analysis, one gets additional
automorphisms from composition with the complex multiplication by cube
roots of unity.

The verification that $\Phi'$ is \'etale and surjective on a connected
component is by the same methods outlined in the proof of
Theorem~\ref{thm:etale}.  The equivalence of the formal moduli functors at
mod $p$-points implies the map is \'etale, and properness is established
through the use of N\'eron models.
\end{proof}

\section{Computation of $\TAF_{GU}(K_1)$}
\label{sec:ellipticquotient}

The analysis of the moduli in the previous section now allows us to
make calculations in homotopy.  

\begin{thm}
\label{thm:genericfricke}
If $p > 3$ and $N \ne 1,3$, the topological automorphic forms
spectrum $\TAF_{GU}(K_1)$ has a factor $E$ whose
homotopy groups are given by the $p$-completion of the subring
\[
E_{2*} \subseteq M_*(\Gamma_0(N))_{\ZZ_p}[\Delta^{-1}]
\]
of the ring of modular forms for $\Gamma_0(N)$ over $\ZZ_p$, consisting of
those elements invariant under an involution.
\end{thm}

\begin{proof}
If $p > 3$, the Adams-Novikov
spectral sequence is concentrated on the zero-line, and collapses to give
an isomorphism 
$$ \pi_{2*} \TMF_0(N)_{p} \cong
M_*(\Gamma_0(N))_{\ZZ_p}[\Delta^{-1}]^\wedge_p. $$ 
In this generic case $N \neq 1,3$, the description of the allowable
isomorphisms asserts that the map $\mc{ M}_0(N)_{\ZZ_p} \to Y$ is a
Galois cover with Galois group $\mb Z/2$, as we have pullback diagram
of Deligne-Mumford stacks 
\xym{
\mc{ M}_0(N)_{\ZZ_p} \coprod \mc{ M}_0(N)_{\ZZ_p} \ar[r] \ar[d] & 
\mc{ M}_0(N)_{\ZZ_p} \ar[d]^{\Phi'} \\
\mc{ M}_0(N)_{\ZZ_p} \ar[r]_{\Phi'} & Y \\
}
where $Y$ is the image component of $\Phi'$ in $\Sh(K_1)$.  
Here the two factors in
the coproduct correspond to the identity morphism and the twist
morphism.  We get a descent spectral sequence with $E_2$-term
\[
H^s(\mb Z/2;\pi_t \TMF_0(N)_{p}) \Rightarrow E_{t-s}.
\]
The group $\mb Z/2$ acts via the Fricke involution on the moduli.  The
lift of the $\mb Z/2$-action to the line bundle of invariant $1$-forms
is the involution obtained by rescaling the natural isogeny by
$a \in \mb Z_p$.  If $p > 2$, the higher cohomology vanishes
and we find that the homotopy of the result consists of the
involution-invariant elements in the ring of modular forms.
\end{proof}

A complete description of this ring rests on a complete description of
the ring of $p$-integral modular forms for $\Gamma_0(N)$, together
with the action of the canonical involution on the $p$-divisible
group.  Thus a more detailed computation requires a case-by-case
analysis.

The rest of the section will be devoted to such an analysis for 
the cases where $N \leq 3$.

\begin{prop}
If $N = 1$, the spectrum  
$\TAF_{GU}(K_1)$
has homotopy groups given by the subring
\[
\pi_* \TAF_{GU}(K_1) \cong \mb Z_p[c_4, c_6^2, \Delta^{-1}]^\wedge_p
\subset \mb Z_p[c_4,c_6,\Delta^{-1}]^\wedge_p
\]
of the $p$-completed ring of modular forms.
\end{prop}

\begin{proof}
In this case, we may (up to natural isomorphism) take the isogeny
$q\co E \to \bar E$ to be the identity map.  As we observed in
Remark~\ref{rmk:Q(i)}, we are simply
restating the $\mb Q(i)$ case of Theorem~\ref{thm:tensor}.
\end{proof}

\begin{thm}
If $p \ne 3$ and $N = 2$, the topological automorphic forms spectrum $\TAF_{GU}(K_1)$
has a factor $E$ whose
homotopy groups are given by the subring
\[
\pi_* E \cong \mb Z_p[q_2, D^{\pm 1}]^\wedge_p \subset
\pi_* \TMF_0(2)_p
\]
of the $p$-completed ring of modular forms of level $2$, 
, where $|q_2| = 4$ and $|D| = 8$.
\end{thm}

\begin{proof}
In the case $N = 2$, the constraint $p > 3$ is forced by
the requirement that $p$ splits in $\mb Q(\sqrt{-2})$ and $p \ne 3$.  
From Theorem
\ref{thm:genericfricke}, we have that $\TAF_{GU}(K_1)$ has a summand 
whose homotopy consists of the invariants in the $p$-completed ring of
modular forms of level $2$ invariant under the involution.

See \cite{markmodular} for a proof of following descriptions.  The
$p$-completed ring of $\Gamma_0(2)$-modular forms with the
discriminant inverted is
\[
\mb Z_p[q_2, q_4, \Delta^{-1}]^\wedge_p,
\]
where $\Delta = q_4^2(16q_2^2 - 64 q_4)$.  The self-map $t$ satisfying
$t^2 = [2]$ that gives rise to the involution is given on
homotopy by
\begin{eqnarray*}
t^*(q_2) &=& -2q_2,\\
t^*(q_4) &=& q_2^2 -4q_4.
\end{eqnarray*}
The involution itself is then given by
\begin{eqnarray*}
w(q_2) &=& q_2,\\
w(q_4) &=& \frac{1}{4} q_2^2 - q_4.\\
\end{eqnarray*}
We formally define the element $r_4$ as $8q_4 - q_2^2$.  The elements
$q_2, r_4,$ and $\Delta^{-1}$ thus generate the ring of modular forms,
and the involution $w$ negates $r_4$.  In this expression, we have the
following identities.
\begin{eqnarray*}
\Delta &=& \frac{1}{8}(q_2^2+r_4)^2 (q_2^2 - r_4)\\
\Delta w(\Delta) &=& \frac{1}{64}(q_2^4-r_4^2)^3.
\end{eqnarray*}

The subring of the ring of modular forms invariant under the 
involution is then generated by
\[
q_2, (q_2^4 - r_4^2), (q_2^4 - r_4^2)^{-1},
\]
as desired.
\end{proof}

\begin{thm}
If $N = 3$, the topological automorphic forms spectrum $\TAF_{GU}(K_1)$
has a summand $E$ whose
homotopy is a subring
\[
\mb Z_p[a_1^6, D^{\pm 1}]^\wedge_p \subset
\pi_* \TMF_0(3)_p
\]
of the $p$-completed ring of modular forms of level $3$, where
$|a_1^6| = |D| = 12$.
\end{thm}

\begin{proof}
In the case $N = 3$, the constraint $p \equiv 1 \mod 3$ is forced by
the requirement that $p$ splits in $\mb Q(\sqrt{-3})$.  The map $\mc{
  M}_0(3)_{\ZZ_p} \to Y$ is a Galois cover, as we have the following
pullback diagram of moduli.
\xym{
\coprod^{6} \mc{ M}_0(3)_{\ZZ_p} \ar[r] \ar[d] & 
\mc{ M}_0(3)_{\ZZ_p}\ar[d]^{\Phi'} \\
\mc{ M}_0(3)_{\ZZ_p} \ar[r] & Y \\
}
Here $Y$ is the image component of $\Sh(K_1)$.  The six factors in the
coproduct correspond to compositions of the action of the third roots
of unity and the involution.  The involution commutes with the action
of the roots of unity.  We get a descent spectral sequence of the form
\[
H^s(\mb Z/6;\pi_t \TMF_0(3)) \Rightarrow E_{t-s},
\]
As $p > 3$, the higher cohomology vanishes.  In this case, the
ring of modular forms for level 3 structures, together with the Fricke
involution, is known; see Mahowald and Rezk \cite{mahowaldrezk}.  We
list the result at primes away from 6.
\[
\TMF_0(3)[1/6]_* \cong \mb Z [1/6][a_1^2, a_1 a_3, a_3^2, \Delta^{-1}]
\]
where $\Delta = a_1^3 a_3^3 - 27 a_3^4$.  The self-map $t$ satisfying
$t^2 = [3]$ that gives rise to the involution is given on
homotopy by
\begin{eqnarray*}
t^*(a_1^2) &=& -3a_1^2,\\
t^*(a_1 a_3) &=& \frac{1}{3} a_1^4 - 9 a_1 a_3,\\
t^*(a_3^2) &=& -\frac{1}{27}a_1^6 + 2 a_1^3 a_3 - 27 a_3^2.\\
\end{eqnarray*}
The involution itself is then given by
\begin{eqnarray*}
w(a_1^2) &=& a_1^2,\\
w(a_1 a_3) &=& \frac{1}{27} a_1^4 -  a_1 a_3,\\
w(a_3^2) &=& \frac{1}{27^2}a_1^6 - \frac{2}{27} a_1^3 a_3 + a_3^2.\\
\end{eqnarray*}
We formally define the element $d_2$ as $(54 \frac{a_3}{a_1} - a_1^2)$.
The elements $a_1^2$ and $d_2$ thus generate a larger ring where the
involution $w$ negates $a_1^2$ and negates $d_2$.  In this
expression, we have the identity
\[
\Delta w(\Delta) = \frac{1}{2^8 3^{18}}(a_1^6 - a_1^2 d_2^2)^4 =
\frac{1}{2^8 3^{18}} D^4.
\]

The subring of the ring of modular forms invariant under the
involution is then generated by elements $a_1^{2k} d_2^{2l} D^{-m}$,
where $2k \geq 2l$.  This subring is generated by the elements
\[
a_1^2, D, D^{-1}.
\]
As in Theorem \ref{thm:tensor}, the third root of unity $\omega$ acts
on modular forms of weight $k$ by multiplication by $\omega^k$.
Therefore, the subring of elements invariant under the action of $\mb
Z/3$ consists precisely of those elements of total degree divisible by
$6$.  This subring is generated by the algebraically independent
elements $a_1^{6}$ and $D$, together with $D^{-1}$.
\end{proof}

\bibliographystyle{amsalpha}
\nocite{*}
\bibliography{absurf4}

\end{document}